\documentclass{article}
\usepackage[utf8]{inputenc}
\usepackage{amsmath, amsthm, amssymb, fdsymbol}
\usepackage{graphicx}
\usepackage{caption}
\usepackage{cite}
\usepackage[noend]{algorithmic}
\theoremstyle{definition}
\newtheorem{definition}{Definition}
\theoremstyle{plain}
\newtheorem{theorem}{Theorem}
\newtheorem{lemma}{Lemma}
\newtheorem{proposition}{Proposition}
\newtheorem{claim}{Claim}

\newtheorem{question}{Question}
\newtheorem{observation}{Observation}
\newtheorem{example}{Example}
\newtheorem{notation}{Notation}
\newtheorem{corollary}{Corollary}
\newenvironment{myproof}[2] {\paragraph{Proof of {#1} {#2} :}}{\hfill$\square$}
\DeclareMathOperator{\weight}{wt}
\begin{document}
\title{Uniquely Pressable Graphs: Characterization, Enumeration, and Recognition}
\author{Joshua N. Cooper and Hays W. Whitlatch}
\maketitle
\begin{abstract}
We consider ``pressing sequences'', a certain kind of transformation of graphs with loops into empty graphs, motivated by an application in phylogenetics.  In particular, we address the question of when a graph has precisely one such pressing sequence, thus answering an question from Cooper and Davis (2015).  We characterize uniquely pressable graphs, count the number of them on a given number of vertices, and provide a polynomial time recognition algorithm.  We conclude with a few open questions.
\end{abstract}
\section{Introduction}
A signed permutation is an integer permutation where each entry is given a sign, plus or minus.  A reversal in a signed permutation is when a subword is reversed and the signs of its entries are flipped.  The primary computational problem of sorting signed permutations by reversals is to find the minimum number of reversals needed to transform a signed permutation into the positive identity permutation.  Hannenhalli and Pevzner famously showed that the unsigned sorting problem can be solved in polynomial time \cite{bergeron2001very, hannenhalli1999transforming} in contrast to the problem of sorting unsigned permutations, which is known to be NP-hard in general \cite{caprara1997sorting}.
At the core of the analysis given in \cite{hannenhalli1999transforming} is the study of ``successful pressing sequences'' on vertex 2-colored graphs. In \cite{cooper2016successful}, the authors discuss the existence of a number of nonisomorphic such graphs which have exactly one pressing sequence, the ``uniquely pressables".  In the context of computational phylogenetics, these graphs correspond to pairs of genomes that are linked by a unique minimum-distance evolutionary history. In this paper we use combinatorial matrix algebra over $\mathbb{F}_2$ to characterize and count the set of uniquely pressable bicolored graphs.
Previous work in the area has employed the language of black-and-white vertex-colored graphs in discussing successful pressing sequences.  For various reasons (such as simplifying definitions and notation), we find it more convenient to replace the black/white vertex-coloring with looped/loopless vertices. Thus, the object of study will be simple pseudo-graphs: graphs that admit loops but not multiple edges (sometimes known as ``loopy graphs''). However, for the purposes of illustration we borrow the convention that the loops of a simple pseudo-graph are drawn as black vertices \cite{bixby2015proving, cooper2016successful}.
Given a simple pseudo-graph $G$, denote by $V(G)$ the vertex set of $G$; $E(G)\subseteq V(G)\times V(G)$, symmetric as a relation, its edge set; and $G[S] = (S, (S\times S) \cap E(G))$ the induced subgraph of a set $S\subset V(G)$. Let $N(v)=N_G(v)=\{w\in V(G):vw\in E(G) \}$ the neighborhood of $v$ in $V(G)$.  Observe that $v\in N(v)$ iff $v$ is a looped vertex.
\begin{figure}[h]
\centering
\captionsetup{justification=centering}
\includegraphics[scale=.30]{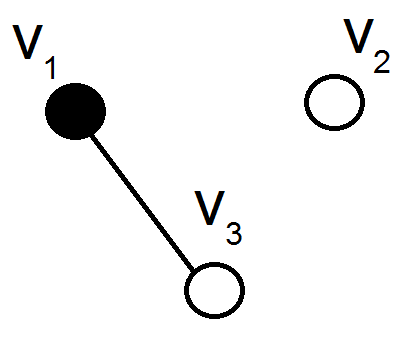}
\caption{A simple pseudo-graph $G=(V,E)$ with $V=\{v_1, v_2, v_3\}$ and  $E=\left\{ v_1 v_1, v_1 v_3\right\}$}
\end{figure}
After this introduction, the discussion is arranged into four sections. In Section 2, we develop some terminology and notation, and give a useful matrix factorization which we refer to as the ``instructional Cholesky factorization'' of a matrix (over $\mathbb{F}_2$), and discuss some of its properties.  In Section 3, we present our main result, Theorem \ref{BIG THEOREM}, which characterizes the uniquely pressable graphs as those whose instructional Cholesky factorizations have a certain set of properties.  In Section 4 we explore some consequences of the main theorem, such as the existence of a cubic-time algorithm for recognizing a uniquely pressable graph, a method for generating the uniquely pressable graphs by iteratively appending vertices to the beginning or end of a pressing sequence, and a counting argument which shows that there exist, up to isomorphism, exactly $(3-(-1)^n)/2 \cdot 3^{\lfloor n/2 \rfloor - 1}$ uniquely pressable graphs on $n$ non-isolated vertices.  In the final section we discuss some open questions in this area.
Before proceeding to Section 2 we list some basic notation for later use.  Other terminology/notation employed below can be found in  \cite{brualdi1991combinatorial} or \cite{diestel2000graph}.
\begin{itemize}
\item We often write $xy$ to represent the edge $\{x,y\}$ for concision.  In particular if $x=y$ then $xy$ is a loop.
\item $[n]:=\{1,2,\ldots,n\}$ and $[k,n]:=\{k, k+1, \ldots, n\}$ for all $k,n \in \mathbb{N}$.
\item When $S=V(G)\setminus\{x\}$ we write the induced subgraph of $G$ on $S$, $G[S]$, as $G-x$.  In general, $G - S$ denotes $G[V(G) \setminus S]$.
\item For integers $x$ and $y$, $x\equiv y \pmod 2$ is abbreviated  as $x\equiv y$. 
\item For a square matrix $M$ with rows and columns identically indexed by a set $X$, for all $x\in X$, $M_{\hat{x}}$ denotes the submatrix of $M$ with row and column $x$ removed.
\item When $\{x_\lambda\}_{\lambda \in \Lambda} \subset \mathbb{F}_2 \cup \mathbb{Z}$, the notation $\sum_{\lambda \in \Lambda}x_\lambda$ denotes addition over $\mathbb{Z}$ of $\overline{x}_\lambda$, where $\overline{x}_\lambda = x_\lambda$ if $x_\lambda \in \mathbb{Z}$ and $\overline{x}_\lambda$ is the least non-negative integer representation of $x_\lambda$ in $\mathbb{Z}$ if $x_\lambda \in \mathbb{F}_2$.    When referring to addition modulo $2$ we use symbols $\oplus$ and $\osum$. For example, if $x_1=x_2= 1 \in \mathbb{F}_2$ and $x_3=3\in \mathbb{Z}$ then 
$$
\sum\limits_{i=1}^{3}x_i=1+1+3=5\quad \textrm{and}\quad\osum\limits_{i=1}^{3}x_i=1\oplus 1 \oplus 1=1.
$$

\end{itemize}
\section{Pressing and Cholesky Roots}
\begin{definition}\label{pressedgraph}
Consider a simple pseudo-graph $G$ with a looped vertex $v \in V(G)$.  ``Pressing $v$'' is the operation of transforming $G$ into $G'$, a new simple pseudo-graph in which $G[N(v)]$ is complemented.  That is, 
$$
V(G')=V(G), \quad E(G')=E(G)\triangle\left(N(v)\times N(v)\right)
$$
We denote by $G_{(v)}$ the simple pseudo-graph resulting from pressing vertex $v$ in $V(G)$ and we abbreviate $G_{(v_1)(v_2)\cdots (v_k)}$ to $G_{(v_1, v_2, \ldots ,v_k)}$. For $k\geq 1$ we abbreviate $(1,2,\ldots,k)$ as $\boldsymbol{k}$ so that 
when $V(G)=[n]$ for some $n\geq k$ then we may simplify $G_{(1,2,\ldots,k)}$ to $G_{\boldsymbol{k}}$.  $G_{\boldsymbol{0}}$ and $G_{()}$ are interpreted to mean $G$.
\end{definition}
Given a simple pseudo-graph $G$, $(v_1, v_2, \ldots,v_j)$ is said to be a \emph{successful pressing sequence} for $G$ whenever the following conditions are met:
\begin{itemize}
\item $\{v_1, v_2, \ldots,v_k\}\subseteq V(G)$,
\item $v_{i}$ is looped in $G_{(v_1, v_2, \ldots,v_{i-1})}$ for all $1\leq i\leq k$, 
\item $G_{(v_1, v_2, \ldots,v_k)}=(V(G), \emptyset)$
\end{itemize}
In other words, looped vertices are pressed one at a time, with ``success'' meaning that the end result (when no looped vertices are left) is an empty graph.  From the definition of ``pressing'' $v$ we see that once a vertex is pressed it becomes isolated and cannot reappear in a valid pressing sequence.  It was shown in \cite{cooper2016successful} that  if $G_{(v_1, v_2, \ldots,v_k)}=(V(G), \emptyset)=G_{(v'_1, v'_2, \ldots,v'_{k'})}$  then $k=k'$, i.e., the length of all successful pressing sequences for $G$ are the same.  We refer to this length $k$ as the \emph{pressing length of $G$}.
\begin{definition}
An \emph{ordered simple pseudo-graph}, abbreviated OSP-graph, is a simple pseudo-graph with a total order on its vertices.  In this paper, we will assume that the vertices of an OSP-graph are subsets of the positive integers under the usual ordering ``$<$''.  An OSP-graph $G$ is said to be \emph{order-pressable} if there exists some initial segment of $V(G)$ that is a successful pressing sequence, that is, if it admits a successful pressing sequence $(v_1, v_2, \ldots,v_k)$ satisfying $v_1<v_2< \cdots < v_k$ and $v_k< v'$ for all $v'\in V(G)\setminus \{v_1, v_2, \ldots,v_k\}$.  
An OSP-graph $G$ is said to be \emph{uniquely pressable} if it is order-pressable and $G$ has no other successful pressing sequence. 
\end{definition}
\begin{lemma} {\label{unique implies full}}
If $G$ is a connected OSP-graph that is uniquely pressable then the pressing length of $G$ is $|V(G)|$.
\end{lemma}
\begin{proof}
Without loss of generality we may assume $V(G)=[n]$.  Assume by way of contradiction that the pressing length of $G$ is $m<n$.  Then the pressing sequence $\boldsymbol{m}=(1,2,\ldots,m)$  is realized by the sequence of graphs 
$$
G, G_{\boldsymbol{1}}, G_{\boldsymbol{2}}, \ldots,G_{\boldsymbol{m}}=([n], \emptyset).
$$
  Let $k=\min\limits_{i\in [m]}\{i\mid \textrm{$G_{\boldsymbol{i}}$ has more than $i$ isolated vertices}\}.$  Then pressing $k$ in $G_{\boldsymbol{k-1}}$ isolates $k$ and at least one more vertex, say $\ell$.  Therefore,
$$
N_{G_{\boldsymbol{k-1}}}(k)\neq \emptyset,\quad N_{G_{\boldsymbol{k-1}}}(\ell)\neq \emptyset, \quad \textrm{ and }\quad N_{G_{\boldsymbol{k}}}(k)=N_{G_{\boldsymbol{k}}}(\ell)= \emptyset.
$$
Then we have the following implications:
$$
v\in N_{G_{\boldsymbol{k-1}}}(\ell)
$$
$$
\Downarrow
$$
$$
v\ell\in E(G_{\boldsymbol{k-1}})\setminus E(G_{\boldsymbol{k}})
$$
$$
\Downarrow
$$
$$
v\ell \in E(G_{\boldsymbol{k-1}})\textrm{ and } v\ell \notin E(G_{\boldsymbol{k}})=E(G_{\boldsymbol{k-1}})\triangle \left(N_{G_{\boldsymbol{k-1}}}(k)\times N_{G_{\boldsymbol{k-1}}}(k) \right)
$$
$$
\Downarrow
$$
$$
v\ell\in E(G_{\boldsymbol{k-1}})\textrm{ and } v\ell \in N_{G_{\boldsymbol{k-1}}}(k)\times N_{G_{\boldsymbol{k-1}}}(k) 
$$
$$
\Downarrow
$$
$$
v \in N_{G_{\boldsymbol{k-1}}}(k)
$$
Hence $N_{G_{\boldsymbol{k-1}}}(\ell)\subseteq N_{G_{\boldsymbol{k-1}}}(k)$.  
However $N_{G_{\boldsymbol{k-1}}}(k) \subseteq N_{G_{\boldsymbol{k-1}}}(\ell)$, since otherwise there exists a $u\in N_{G_{\boldsymbol{k-1}}}(k)\setminus N_{G_{\boldsymbol{k-1}}}(\ell)$, so $\ell u \in E(G_{\boldsymbol{k}})$ because $k \in N_{G_{\boldsymbol{k-1}}}(\ell)$, contradicting the fact that $\ell$ is isolated in $G_{\boldsymbol{k}}$.  It follows that $k$ and $\ell$ are twins in $G_{\boldsymbol{k-1}}$ (i.e., there is an automorphism fixing all vertices except $k$ and $l$), so
$$
(1,2,\ldots,k-1, \ell, k+1,\ldots,\ell-1,k,\ell+1,\ldots,m)
$$
is a successful pressing sequence of $G$ in addition to $\boldsymbol{m}$, contradicting unique pressability.  We conclude that $m\not<n$, that is $m=n$.
\end{proof}
\noindent We say a component of $G$ is trivial if it is a loopless isolated vertex.   
\begin{proposition}{\label{succesful}\cite{cooper2016successful}} A simple pseudo-graph  $G$ admits a successful pressing sequence if and only if every non-trivial component of $G$ contains a looped vertex. \end{proposition}
\begin{corollary}
If $G$ is a uniquely pressable OSP-graph  with at least one edge  then $G$ contains exactly one non-trivial component $C$ and the pressing length of $G$ is $|V(C)|$. 
\end{corollary}
\begin{proof}
Let $G=([n], E)$.  Let $C_1$ and $C_2$ be (possibly distinct) non-trivial connected components of $G$. Let $C_3=G-(V(C_1)\cup V(C_2))$ so that $G$ is the (possibly disjoint) union of $C_1, C_2$ and $C_3$.  As $G$ is uniquely pressable it has unique pressing sequence $\sigma=\boldsymbol{n}$.  Observe that pressing a vertex only makes changes to its closed neighborhood, a set which is contained within a single connected component.  Let $\sigma_i$ be the restriction of $\sigma$ to the vertices of $C_i$, $i=1,2,3$.   Then $\sigma_i$ is a successful pressing sequence for $G[C_i]$. 
If $C_1\neq C_2$ then $G$ is the disjoint union of $G[C_1]$, $G[C_2]$ and $G[C_3]$, where the last one may be empty. Then pressing the vertices of $G[C_2]$ followed by pressing the vertices of $G[C_1]$ followed by pressing the vertices of $G[C_3]$ gives a successful pressing sequence for $G$, contradicting the uniqueness of $\sigma$.    It follows that $C_1=C_2$ and $\sigma_3=\emptyset$, and therefore $G$ contains exactly one non-trivial connected component.  
\end{proof}
This shows that in order to understand uniquely pressable OSP-graphs, it suffices to understand connected, uniquely pressable OSP-graphs.
\begin{notation}
$\mathbf{CUP}_n$ is the  set of connected, uniquely pressable ordered ($<_{\mathbb{N}}$) simple pseudo-graphs on $n$ positive integer vertices.    
\end{notation} 
\begin{definition}
Given an OSP-graph $G=([n], E)$ define the \emph{adjacency matrix} $A=A(G)=(a_{i,j})\in \mathbb{F}_2^{n\times n}$  by 
$$
a_{i,j}=\begin{cases}1 & \textrm{ if  $ij\in E$,}\\
0 & \textrm{ otherwise.}
\end{cases}.
$$
Note that $A(G)$ is always symmetric.  Previous work in the area refers to such matrices as  \emph{augmented} adjacency matrices as the diagonal entries are nonzero where the vertices are colored black; since we have used looped vertices instead, the term ``augmented'' is not necessary.
Define the \emph{instructional Cholesky root} of $G$, denoted $U=U(G)=(u_{i,j})\in \mathbb{F}_2^{n\times n}$, by 
$$
u_{i,j}=\begin{cases}1 & \textrm{ if $i\leq j$ and $j \in N_{G_{\boldsymbol{i-1}}}(i)$,}\\
0 & \textrm{ otherwise.}
\end{cases}.
$$
Observe that the $j^{th}$ row of $U$ is given by the $j^{th}$ row of the adjacency matrix of $G_{\boldsymbol{j-1}}$, and that $u_{i,j}=1$ precisely when the act of pressing $i$ during a successful pressing sequence of $G$ flips the state of $j$.  Thus, $U$ provides detailed ``instructions'' on how to carry out the actual pressing sequence.  In Proposition \ref{U is Cholesky} we justify use of the name Cholesky.
\end{definition}
\begin{definition}
For an order-pressable graph $G=([n],E)$ with instructional Cholesky root $U=(u_{i,j})$ we define the \emph{dot product of two vertices} $i,j \in V(G)$ as the dot product over $\mathbb{F}_2$ of the $i^{th}$ and $j^{th}$ columns of $U$: 
$$
 \langle i,j \rangle_G=\osum\limits_{t=1}^{n}u_{t,i}u_{t,j}.
$$
\noindent We define the \emph{(Hamming) weight of a vertex} $j\in [n]$ by 
$$
\weight_G(j)=\sum\limits_{t=1}^{n}u_{t,j}
$$
 and observe that $\weight_G(j)\equiv{\langle j,j \rangle}_G$.
\end{definition}
\begin{figure}[h]
\centering
\captionsetup{justification=centering}
\includegraphics[scale=.30]{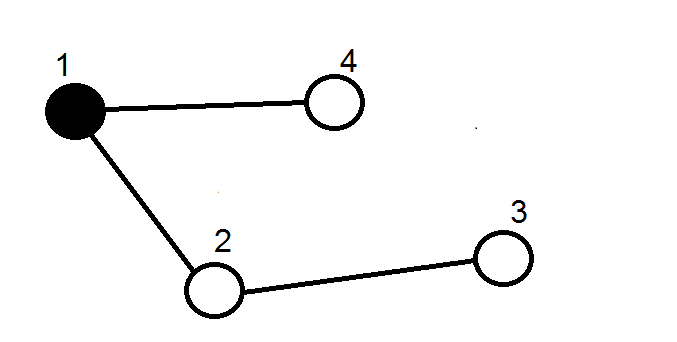}
\caption{\\$\weight(1)=1$, $\weight(2)=2$, $\weight(3)=2$, $\weight(4)=4$ \\ $\langle 1,2\rangle =1$, $\langle 1,3\rangle =0$, $\langle 1,4\rangle =1$, $\langle 2,3\rangle =1$, $\langle 2,4\rangle =0$, $\langle 3,4\rangle =0$}
\end{figure}
\begin{definition}
The \emph{weight of a column $C$} in a matrix $M$, written $\weight_M(C)$, is the sum of the entries  in column $C$ (again, as elements of $\mathbb{Z}$).
\end{definition}
Observe that if $U$ is the instructional Cholesky root of $G$ then the column weights of $U$ correspond to the vertex weights of $G$.
\begin{proposition}\label{U is Cholesky}
If $G=([n],E)$ is an OSP-graph with successful pressing sequence $1,2,\ldots, k$, adjacency matrix $A$, and instructional Cholesky root $U$ then 
$$
U^TU=A.
$$
\end{proposition}
\begin{proof}
Let $U^TU=B=(b_{i,j})$ and $A=(a_{i,j})$.  Observe that $b_{i,j}$ is the result (modulo $2$) of dotting the $i^{th}$ and $j^{th}$ columns of $U$.  Hence 
$$
b_{i,j}= \langle i,j \rangle_U.
$$
\noindent For $i,j\in [n]$ let 
$$
S_{i,j}=\{t \in [k] \colon ij \in E(G_{\boldsymbol{t-1}}) \triangle  E(G_{\boldsymbol{t}}) \}
$$
 and 
$$
T_{i,j}=\{t \in [k] \colon ti \in E(G_{\boldsymbol{t-1}}) \textrm{ and } tj \in E(G_{\boldsymbol{t-1}}) \}.
$$
Observe that $S_{i,j}$ lists the times during the pressing sequence that the pressed vertex results in the state of edge $ij$ being flipped.  This occurs if and only if  both $i$ and $j$ are in the neighborhood of the vertex being pressed.  Hence $S_{i,j}=T_{i,j}$.
The state of the edge/non-edge $ij$ in $G_{\boldsymbol{k}}$ is determined by its original state in $G$ and by the number of times the state of the edge/non-edge $ij$ was flipped during the pressing sequence.  However, $G_{\boldsymbol{k}}=([n], \emptyset)$ so $ ij\notin E(G_{\boldsymbol{k}})$ and therefore the number of times that the state of the edge/non-edge $ij$ is flipped during the pressing sequence must agree in parity to with the original state of the edge/non-edge $ij$.  It follows that $|S_{i,j}|\equiv a_{i,j}$.
On the other hand $T_{i,j}=\{t \in [n]\colon u_{t,i}=u_{t,j}=1\}$ list the common $1$'s in columns $i$ and $j$ of the instructional Cholesky root.  Hence $|T_{i,j}|$ has the same parity as dotting the $i^{th}$ and the $j^{th}$ column of $U$.  It follows that 
$$
b_{i,j}= \langle i,j \rangle_G\equiv \vert T_{i,j}\vert=\vert S_{i,j}\vert\equiv a_{i,j}.
$$
Since the matrix entries are elements of $\mathbb{F}_2$,  we have  
$$
b_{i,j}=a_{i,j} 
$$
for $i,j\in [n]$ and therefore $U^TU=A$.
\end{proof}
\begin{observation} \textrm{ }\newline
Given an OSP-graph $G$ with adjacency matrix $A$ and instructional Cholesky root $U$,
$$
ij \in E(G) \textrm{ if and only if } \langle i,j \rangle_G=1,
$$
since $a_{i,j}$ is the result of dotting the $i^{th}$ and $j^{th}$ columns of $U$.  In particular 
$$
i \textrm{ is looped in $G$ if and only if } \weight_G(i)\equiv 1.
$$
\end{observation}
In the theory of complex matrices, decompositions of the form $A = U^T U$ are known as ``Cholesky'' factorizations, so we repurpose this terminology here.  While a symmetric full-rank matrix over $\mathbb{F}_2$ has a unique Cholesky decomposition (see \cite{cooper2016successful}), a matrix $M\in \mathbb{F}_2^{n\times n}$ of less than full rank may have more than one Cholesky decomposition.  On the other hand, the  adjacency matrix $A$ of an OSP-graph $G$ with successful pressing sequence $1,2,\ldots,k$ has a unique instructional Cholesky root $U$ as the first $k$ rows are determined by the sequence of graphs $G, G_{\boldsymbol{1}}, G_{\boldsymbol{2}}, \ldots, G_{\boldsymbol{k-1}}$ and the remaining rows (should they exist) are all zero.  Throughout the paper we will take advantage of this by referring interchangeably to a  pressable OSP-graph $G$, its (ordered) adjacency matrix $A$, and its instructional Cholesky root $U$.
\begin{example}
Consider $M\in \mathbb{F}_2^{5\times 5}$ given by 
$$
M=\begin{bmatrix}
1 & 0 & 0 & 0 & 1\\
0 & 1 & 0 & 1 & 0\\ 
0 & 0 & 0 & 0 & 0\\ 
0 & 1 & 0 & 1 & 0\\ 
1 & 0 & 0 & 0 & 1\\ 
\end{bmatrix}.
$$
The (unique) instructional Cholesky root of $M$ is 
$$
U=\begin{bmatrix}
1 & 0 & 0 & 0 & 1\\
0 & 1 & 0 & 1 & 0\\ 
0 & 0 & 0 & 0 & 0\\
0 & 0 & 0 & 0 & 0\\
0 & 0 & 0 & 0 & 0\\
\end{bmatrix}.
$$
  The following matrices also offer Cholesky factorizations for $M$:
$$
\begin{bmatrix}
1 & 0 & 0 & 0 & 1\\
0 & 1 & 0 & 1 & 0\\ 
0 & 0 & 0 & 0 & 0\\
0 & 0 & 0 & 0 & 1\\
0 & 0 & 0 & 0 & 1\\
\end{bmatrix},\begin{bmatrix}
1 & 0 & 0 & 0 & 1\\
0 & 1 & 0 & 1 & 0\\ 
0 & 0 & 0 & 0 & 1\\
0 & 0 & 0 & 0 & 1\\
0 & 0 & 0 & 0 & 0\\
\end{bmatrix},\begin{bmatrix}
1 & 0 & 0 & 0 & 1\\
0 & 1 & 0 & 1 & 0\\ 
0 & 0 & 0 & 0 & 1\\
0 & 0 & 0 & 0 & 0\\
0 & 0 & 0 & 0 & 1\\
\end{bmatrix},
$$
$$
\begin{bmatrix}
1 & 0 & 0 & 0 & 1\\
0 & 1 & 0 & 1 & 0\\ 
0 & 0 & 0 & 1 & 0\\
0 & 0 & 0 & 1 & 0\\
0 & 0 & 0 & 0 & 0\\
\end{bmatrix} \textrm{ and } \begin{bmatrix}
1 & 0 & 0 & 0 & 1\\
0 & 1 & 0 & 1 & 0\\ 
0 & 0 & 0 & 1 & 1\\
0 & 0 & 0 & 1 & 1\\
0 & 0 & 0 & 0 & 0\\
\end{bmatrix}.
$$
\end{example}
\begin{definition}
Consider a pressable OSP-graph $G=(V,E)$  where $V=\{v_i\}_{i \in [n]}$ has the order implied by its indexing.  For each $j \in [n]$ we say that $v_j$ has \emph{full weight in $G$} provided 
$$
\weight_G(v_j)= j.
$$
  In particular if  $V(G)=[n]$ under the usual ordering then vertex $j$ has full weight if and only if $\weight_G(j)=j$, if and only if 
$$
j\in N_{G_{\boldsymbol{i-1}}}(i)\quad \textrm{ for all } i \in [j]. 
$$
\end{definition}
\noindent The following notation will be used to simplify inductive arguments. 
\begin{notation}
For a given OSP-graph $G=([n], E)$ with looped vertex $j$ denote by $G^{(j)}=G_{(j)}-j$  the result of pressing vertex $j$ and then deleting it from the vertex set.  Furthermore we let $G^{\boldsymbol{j}}$ denote the result of pressing and deleting vertices $1,2,\ldots, j$ in order from $G$.  
\end{notation}
\begin{lemma}\label{pressed graph has nice cholesky}
If $G\in \mathbf{CUP}_n$ has instructional Cholesky root $U$ then $G^{\boldsymbol{1}} \in \mathbf{CUP}_{n-1}$ and the instructional Cholesky root of $G^{\boldsymbol{1}} $ is $U_{\hat{1}}$.
\end{lemma}
\begin{proof}
Let  $G=([n], E) \in \mathbf{CUP}_n$ with instructional Cholesky root $U$. The unique successful pressing sequence of $G$ is $\boldsymbol{n}$ which is realized by 
$$
G, G_{\boldsymbol{1}}, G_{\boldsymbol{2}}, \ldots, G_{\boldsymbol{n}}
$$
 and hence $G_{\boldsymbol{1}}$ admits a successful pressing sequence: $2,3,\ldots,n$.  Furthermore, if $(v_1,v_2, \ldots, v_{n-1})$ is a successful pressing sequence of $G_{\boldsymbol{1}}$, then $(1,v_1,v_2, \ldots, v_{n-1})$ is a successful pressing sequence of $G$.  By uniqueness it follows that $v_i=i+1$ for each $i\in [n-1]$.  Then $G_{\boldsymbol{1}}$ admits exactly one successful pressing sequence $2,3,\ldots,n$, and therefore so does $G^{\boldsymbol{1}}$.  It follows that $G^{\boldsymbol{1}} \in \mathbf{CUP}_{n-1}$.
Let $V$ be the instructional Cholesky root of $G^{\boldsymbol{1}}$.  The first row of $G^{\boldsymbol{1}}$ is given by the neighborhood of $2$ in $G_{\boldsymbol{1}}$ and in general the $j^{th}$ row of $V$ is given by $N_{G^{\boldsymbol{1}}(2,3,\ldots,j)}(j+1)$. However
$$
N_{G^{\boldsymbol{1}}(2,3,\ldots,j)}(j+1)=N_{G_{\boldsymbol{j}}}(j+1)
$$
  for each $j\in [n-1]$.  Therefore the $j^{th}$ row of $V$ is the $(j+1)^{th}$ row of $U$ with the first entry deleted, since $1$ is not a vertex in $G^{\boldsymbol{1}}$.  $V$ is the principal submatrix of $U$ restricted to rows and columns $2,3,\ldots,n$.
\end{proof}
\begin{proposition}{\label{cooptastic} \cite{cooper2016successful}}
An OSP-graph $G=([n], E)$ has pressing sequence $\boldsymbol{n}$ if and only if every leading principal minor of its adjacency matrix is nonzero.
\end{proposition}
\begin{lemma}\label{principal submatrix lemma}
Let $G=([n], E)\in \mathbf{CUP}_n$ with instructional Cholesky root $U$ and let $H=G-n$ be the induced subgraph of $G$ on $[n-1]$.  Then $H \in \mathbf{CUP}_{n-1}$ and the instructional Cholesky root of $H$ is $U_{\hat{n}}$.
\end{lemma}
\begin{proof}
If $n=1$ then $H=(\emptyset, \emptyset)$ which has only the empty sequence as a successful pressing sequence and its instructional Cholesky root is the empty matrix. Let $n>1$.  Observe that $N_H(1)=N_G(1)-\{n\}$ so $1$ is a looped vertex in $H$ and therefore may be pressed to obtain $H_{\boldsymbol{1}}$.
For all $j\in [n-1]$:
\begin{eqnarray}
N_{H_{\boldsymbol{1}}}(j)&=&\begin{cases} N_{H}(j)\triangle N_{H}(1),&1j\in E(H)\\  N_{H}(j),& 1j\notin E(H)\\ \end{cases}\nonumber \\
&=&\begin{cases} \left(N_{G}(j)\triangle N_{G}(1)\right)-\{n\},& 1j\in E(H)\leftrightarrow 1j\in E(G)\\
N_{G}(j)-\{n\}, & 1j\notin E(H)\leftrightarrow 1j\notin E(G)\\ \end{cases}\nonumber \\
&=&\begin{cases} N_{G_{\boldsymbol{1}}}(j)-\{n\},&  1j\in E(G)\\
N_{G_{\boldsymbol{1}}}(j)-\{n\}, & 1j\notin E(G)\\
\end{cases}\nonumber 
\end{eqnarray}
Assume $H_{\boldsymbol{i}}=G_{\boldsymbol{i}}-n$ for some $1\leq i<n-1$.  Then $i+1$ is looped in $G_{\boldsymbol{i}}$, implying that it is looped in  $H_{\boldsymbol{i}}$, so for all $j\in [n-1]$:
\begin{eqnarray}
N_{H_{\boldsymbol{i+1}}}(j)&=&\begin{cases} N_{H_{\boldsymbol{i}}}(j)\triangle N_{H_{\boldsymbol{i}}}(i+1),& \{j,i+1\} \in E(H_{\boldsymbol{i}})\\
N_{H_{\boldsymbol{i}}}(j),& \{j,i+1\} \notin E(H_{\boldsymbol{i}})\\ \end{cases}\nonumber \\
&=&\begin{cases} \left(N_{G_{\boldsymbol{i}}}(j)\triangle N_{G_{\boldsymbol{i}}}(i+1)\right)-\{n\},&  \{j,i+1\} \in E(G_{\boldsymbol{i}})\\
N_{G_{\boldsymbol{i}}}(j)-\{n\}, &  \{j,i+1\} \notin E(G_{\boldsymbol{i}})\\
\end{cases}\nonumber \\
&=&\begin{cases} N_{G_{\boldsymbol{i+1}}}(j)-\{n\},&  \{j,i+1\} \in E(G_{\boldsymbol{i}})\\
N_{G_{\boldsymbol{i+1}}}(j)-\{n\}, & \{j,i+1\} \notin E(G_{\boldsymbol{i}})\\
\end{cases}\nonumber 
\end{eqnarray}
By induction it follows that $\boldsymbol{n-1}$ is a valid pressing sequence for $H$ and $H_{\boldsymbol{i}}=G_{\boldsymbol{i}}-\{n\}$ for all $i\in [n-1]$.  We proceed to show that $\boldsymbol{n-1}$ is the only successful pressing sequence for $H$.
Let $A$ be the adjacency matrix of $G$ (under the ordering $\boldsymbol{n}$) and let $U$ be its instructional Cholesky root. Let $\sigma=(v_1, v_2, \ldots, v_{n-1})$ be a valid pressing sequence for $H$ and let $\tau=(v_1, v_2, \ldots, v_{n-1}, n)$. 
Let $P$ be the permutation matrix that encodes $\tau$.  Then $A_{\hat{n}}$ is the adjacency matrix of $H$ under the usual ordering $<$ and $P_{\hat{n}}A_{\hat{n}}{P_{\hat{n}}}^T$ is the adjacency matrix of $H$ under the ordering given by $\sigma$.  Let $V$ be the instructional Cholesky root of $H$ under $\sigma$.
Observe that by Proposition \ref{cooptastic}, $\det(A)\neq 0$ and so
$$
\det(PAP^T)=\det(P)\det(A)\det(P^T)=\det(A)\neq 0.
$$
Furthermore 
$$
PAP^T=\left[\begin{array}{c|c}
\begin{matrix} &&\\&   \mbox{{$P_{\hat{n}}$}}&\\
    &&\\ \end{matrix}  & \begin{matrix} 0 \\ \vdots  \\ 0\\
\end{matrix}\\ \hline 
\begin{matrix} 0 &\cdots & 0 \\
\end{matrix}  &  1\\
\end{array}\right]
\left[\begin{array}{c|c}
\begin{matrix} &&\\&\mbox{{$A_{\hat{n}}$}}&\\
    &&\\ \end{matrix}  & \begin{matrix} * \\ \vdots  \\ *\\
\end{matrix}\\ \hline 
\begin{matrix} * &\cdots & * \\
\end{matrix}  &  1\\
\end{array}\right]
\left[\begin{array}{c|c}
\begin{matrix} &&\\&   \mbox{{$P_{\hat{n}}$}}&\\
    &&\\ \end{matrix}  & \begin{matrix} 0 \\ \vdots  \\ 0\\
\end{matrix}\\ \hline 
\begin{matrix} 0 &\cdots & 0 \\
\end{matrix}  &  1\\
\end{array}\right]^T
$$
$$
=
\left[\begin{array}{c|c}
\begin{matrix} &&\\&   \mbox{{$P_{\hat{n}}A_{\hat{n}}{P_{\hat{n}}}^T$}}&\\
    &&\\ \end{matrix}  & \begin{matrix} * \\ \vdots  \\ *\\
\end{matrix}\\ \hline 
\begin{matrix} * &\cdots & * \\
\end{matrix}  &  *\\
\end{array}\right]
=
\left[\begin{array}{c|c}
\begin{matrix} &&\\&   \mbox{{$V^T V$}}&\\
    &&\\ \end{matrix}  & \begin{matrix} * \\ \vdots  \\ *\\
\end{matrix}\\ \hline 
\begin{matrix} * &\cdots & * \\
\end{matrix}  &  *\\
\end{array}\right]
$$
Recall that $H$ has $\boldsymbol{n-1}$ as a successful pressing sequence and so every successful pressing sequence must have length $n-1$.  It follows that all the diagonal entries of (upper/lower-triangular matrices) $V$ and $V^T$ must be $1$, implying that every leading principal minor of $PAP^T$ is non-zero.  By Proposition \ref{cooptastic}, $\tau$ is a successful pressing sequence for $G$.  By uniqueness $\tau=\boldsymbol{n}$ and hence $\sigma=\boldsymbol{n-1}$. We may conclude that $H\in \mathbf{CUP}_{n-1}$.
\end{proof}
\begin{corollary}\label{principal submatrices}
Let $G\in \mathbf{CUP}_n$ with instructional Cholesky root $U$.  Then any principal submatrix of $U$ on $k$ consecutive rows and columns is the instructional Cholesky root of a $\mathbf{CUP}_k$ graph.
\end{corollary}
\begin{proof}
Follows by iteratively applying Lemmas \ref{pressed graph has nice cholesky} and \ref{principal submatrix lemma}.
\end{proof}
\begin{corollary}\label{superdiagonals are ones}
If $U$ is the instructional Cholesky root of $G\in\mathbf{CUP}_n$ then $U$ must have all $1$'s on the main diagonal and super-diagonal.  
\end{corollary}
\begin{proof}
Let $H=([2], E) \in \mathbf{CUP}_2$.  Since it is connected, $\{1,2\} \in E$; since it is order-pressable, $1$ must be looped; and since it is uniquely pressable, $N_H(1)\neq N_H(2)$.  Therefore $\mathbf{CUP}_2=\{([2], \{\{1,1\},\{1,2\}\})\}$ which corresponds to instructional Cholesky root $\begin{bmatrix} 1& 1\\ 0&1\\\end{bmatrix}$.  The result holds by application of Corollary \ref{principal submatrices}.
\end{proof}
\section{Characterizing Unique Pressability}
\begin{definition}
For an upper-triangular matrix $M\in \mathbb{F}_2^{n\times n}$ with columns $C_1$, $C_2$, $\ldots$, $C_n$ with respective column weights $w_1, w_2, \ldots, w_n$, we say:
\begin{itemize}
\item $C_j=(c_{1,j}, c_{2,j}, \ldots, c_{n,j})^T$ has {\em Property 1}  if $\begin{cases} c_{i,j}= 1, &  j-w_j < i\leq j\\
c_{i,j}= 0, &  \textrm{ otherwise}\\
\end{cases}$.
\item $M$ has {\em Property 1} if each of its columns have Property 1.
\item $M$ has {\em Property 2} if $1=w_1\leq w_2 \leq \cdots \leq w_n$
\item $M$ has {\em Property 3} if $w_i>2$ implies $w_{i+2}>w_i$, for $i\in [n-2]$.
\item $M$ has {\em Property 4} if, whenever some non-initial column has odd weight, then it must have full weight and so must each column to its right.
\end{itemize}
In other words, Property 1 is the condition that the nonzero entries in each column are consecutive and end at the diagonal; Property 2 is the condition that the weights of the columns are nondecreasing; and Property 3 is the condition that any column must have weight greater than that of the column two indices to its left if the latter has weight more than $2$. Note also that Property 4 implies that any even-indexed column must have even weight; otherwise, it would have full weight, i.e., weight equal to the column index, which is even, a contradiction.  Let $\mathcal{M}_n=\{M\in \mathbb{F}_2^{n\times n}\mid M \textrm{ satisfies Properties 1, 2, 3 and 4} \}$.     Observe that if $M \in \mathcal{M}_n$ then $M_{\hat{n}}\in \mathcal{M}_{n-1}$.
\end{definition}
\begin{lemma}\label{properties are inherited}
Let $n>1$ and $M\in \mathbb{F}_2^{n\times n}$ with columns and rows  indexed by $1,2,\ldots,n$.
If $M\in \mathcal{M}_n$ then $M_{\hat{1}}\in \mathcal{M}_{n-1}$.
\end{lemma}
\begin{proof}
Let $M=(c_{i,j})_{i,j\in [n]}$ so that   $M_{\hat{1}}=(c_{i,j})_{2\leq i,j\leq n}$. Let $w_1, w_2, .., w_n$ be the column weights of $M$. Let the columns and rows of $M_{\hat{1}}$ be $C_2, \ldots, C_n$ with  weights $w'_2, \ldots, w'_n$, respectively.  Observe that $w'_j=w_j-c_{1,j}$ for each $2\leq j\leq n$.  It is immediate that $M_{\hat{1}}$ inherits Property 1 from $M$. 
By Property 4  we know that the second column of $M$ (as well as any even-indexed column of $M$) has even weight, it follows that $w_2=2$ and so $w'_2=2-c_{1,2}=1$.  Suppose towards a contradiction $w'_i>w'_{i+1}$ for some $2\leq i\leq n-1$. Then 
$$
w'_{i}+1>w'_{i+1}+1\geq w'_{i+1}+c_{1,i+1} = w_{i+1}\geq w_i\geq w'_i
$$
 and so $w'_i=w_i$.  Then 
$$
w'_{i+1}<w'_i=w_i\leq w_{i+1}= w'_{i+1}+c_{1,i+1}\leq w'_{i+1}+1
$$
 and so $w_{i+1}=w'_{i+1}+1$.  Hence we have 
$$
w_{i+1}=w'_{i+1}+1<w'_{i}+1=w_i+1\quad \Rightarrow \quad w_{i+1}\leq w_i.
$$
It follows that $w_{i+1}= w_i$ and therefore column $i+1$ does not have full weight.  By Property 1 of $M$ we have $c_{1,i+1}=0$ and therefore  $w'_{i+1}=w_{i+1}$, a contradiction.  It follows that $M_{\hat{1}}$ has Property 2.

Suppose now that $2<w'_j$ for some $2\leq j\leq n-2$. Then $w_j\geq w'_j>2$ so  $w_{j+2}>w_j$ by Property 3 of $M$. 
If $w_j= w_{j+2}-1$ then either column $j$ or column $j+2$ has odd weight which implies $w_{j+2}=j+2$ by Property 4 of $M$.  But then $w_j=j+1$ which is not possible.  Therefore,   
$$
w_j<w_{j+2}-1  
$$
and 
$$
w'_j\leq w_j< w_{j+2}-1\leq w'_{j+2}
$$
which shows that $M_{\hat{1}}$ has Property 3.
To show Property 4 suppose $w'_k\equiv 1$ for some $k>2$ ($2$ is the initial column of $M_{\hat{1}}$).  Observe that
$$
w'_k+1\geq w_k\geq w'_k.
$$
If $w_k=w'_k$ then $w_k\equiv 1$ and not full weight, contradicting Property 4.  Hence 
$$
w_k=w'_k+1
$$
 and $c_{1,k}=1$.  It follows from Property 1 that $w_k=k>2$. Hence 
$$
w'_k>1
$$
 and since $w'_k\equiv 1$ then 
$$
w'_k\geq 3\quad \textrm{ and }\quad w_k\geq 4. 
$$
Applying Property 3 of $M$ 
$$
k+2\geq w_{k+2}\geq w_k+1=k+1.
$$
Since $1\equiv w'_k=k-1$ we have that $w_{k+2}$ is the weight of an even-indexed column, and so $w_{k+2}\equiv 0$ and 
$$
w_{k+2}=k+2.
$$
By arguing inductively, for all $j\in \left[\lfloor (n-k)/2 \rfloor \right]$: 
$$
0\equiv k+2j\geq w_{k+2j}\geq w_{k+2(j-1)}+1=k+2j-1\equiv 1
$$
hence
$$
w_{k+2j}=k+2j.
$$
It follows that for all $j\in \left[\lfloor (n-k)/2 \rfloor \right]$:
$$
w'_{k+2j}=(k+2j)-c_{1,k+2j}=k+2j-1.
$$
Furthermore, for all $j\in \left[0,\lceil (n-k-1)/2 \rceil \right]$: 
$$
w_{k+2j+1}\geq w_{k+2j}=k+2j 
$$
and so 
$$
w_{k+2j+1}=k+2j+c_{1, k+2j+1}.
$$
Therefore, for all $j\in \left[0,\lceil (n-k-1)/2 \rceil \right]$:
$$
w'_{k+2j+1}= w_{k+2j+1}-c_{1, k+2j+1}=k+2j.
$$
It follows that, in $M_{\hat{1}}$, all the columns of index at least $k$ have full weight and therefore $M_{\hat{1}}$ has Property 4.
\end{proof}
Observe that the previous lemma can be extended to any matrix $M \in \mathcal{M}_n$ by relabeling the rows and columns.  We now proceed to our main theorem, which characterizes the set $\mathbf{CUP}_n$.  This in turn provides a characterization of all the uniquely pressable simple pseudo-graphs (up to isomorphism), since the unique non-trivial, connected component of a simple pseudo-graph can always be relabeled to be a $\mathbf{CUP}$ graph.
\begin{notation}
For an OSP-graph $G$ let 
$$
\mathcal{L}(G)=\{v\mid v \in V(G) \textrm{ is a looped vertex}\}.
$$
\end{notation}
\begin{theorem}\label{BIG THEOREM}
Let $G=([n],E)$ with instructional Cholesky root $U$. Then  $G  \in \mathbf{CUP}_n$  if and only if $U\in \mathcal{M}_n$.
\end{theorem}
\begin{proof}
For $n=1$ the conditions of $\mathcal{M}_1$ are only met by $G=([1], \{(1,1)\})$ which in turn is the only full-length  uniquely pressable OSP-graph on vertex set $[1]$.  Let $n>1$ and assume towards an inductive argument that the statement holds for $n-1$. 
\noindent We begin by showing sufficiency,  that is if  $U\in \mathcal{M}_n$ then $G  \in \mathbf{CUP}_n$.
Choose and fix $U=(u_{i,j}) \in \mathcal{M}_n$.  Let  $G=([n], E)$ be the OSP-graph with instructional Cholesky root $U$.  By Properties $1$ and $2$, $u_{i,i}=1$ for each $i\in [n]$.  This implies that vertex $i$ is looped in $G_{\boldsymbol{i-1}}$ for each $i \in [n]$.  It follows that $\boldsymbol{n}$ is a successful pressing sequence for $G$. We will show it is the only successful pressing sequence for $G$.  Fix a successful pressing sequence $\sigma=\sigma_1, \ldots, \sigma_n$ for $G$. 
If $\sigma_1=1$ then $G^{(\sigma_1)}$ has adjacency matrix 
$$
A(G^{\boldsymbol{1}})=U^T_{\hat{1}}U_{\hat{1}}.
$$
By Lemma \ref{properties are inherited}, $U_{\hat{1}}\in \mathcal{M}_{n-1}$ and therefore $G^{(\sigma_1)}\in \mathbf{CUP}_{n-1}$ by the inductive hypothesis.  Hence $G^{(\sigma_1)}$ has exactly one pressing sequence, and, since $G^{(\sigma_1)}=G^{\boldsymbol{1}}$, the sequence is $(2,3,\ldots,n)$.  We may conclude that if $\sigma_1=1$ then $\sigma =\boldsymbol{n}$.
Assume, by way of contradiction, that $\sigma_1=t> 1$.  We will show that $G^{(t)}$  contains a non-trivial loopless component and therefore is not pressable.  
Since $t$ is a looped vertex it must have odd weight, and therefore full weight by Property 4.   Let 
$$
k=\min\limits_{2\leq i\leq n}\{i\mid  \weight_G(i)\equiv 1\}.
$$
\noindent Let 
$$
L=[2,k-1],\quad  R=[k,n],\quad  \overline{L}=L\cup \{1\},\quad  \textrm{and} \quad \overline{R}=R\cup \{1\}.
$$
By Property 4 of $U$, all the vertices in $\overline{R}$ have full weight.  For all $i\in [n]$ and $r\in \overline{R}$:
$$
\langle i,r \rangle_G=\osum_{k=1}^{n}u_{k,i}u_{k,r}=\osum_{k=1}^{r}(u_{k,i}\cdot 1)\oplus \osum_{k=r+1}^{n}(u_{k,i}\cdot 0)=\min\left\{\weight_G(i), \weight_G(r) \right\}.
$$
By Property 4, if $\weight_G(i)\equiv 1$ then $i \in \overline{R}$.  It follows that for all $r\in \overline{R}$,
$$
N_G(r)=\begin{cases}
\mathcal{L}(G)\cap [r-1], & \textrm{ if $r\equiv 0$}\\
\mathcal{L}(G)\cup [r, n], & \textrm{ if $r\equiv 1$}\\
\end{cases}\quad  \subseteq\quad  \mathcal{L}(G)\subseteq \overline{R}.
$$
\noindent Then 
$$
\mathcal{L}\left(G^{(t)}\right) =\mathcal{L}(G)\triangle N_{G}(t)\subseteq \overline{R}. 
$$
However, $\langle 1,t \rangle = 1$ because $t$ has full weight, so $1\in \mathcal{L}(G)\cap N_{G}(t)$ and
$$
\mathcal{L}\left(G^{(t)}\right)\subseteq R. 
$$
Similarly, for $r\in R$
$$
N_{G^{(t)}}(r)=N_G(r)\triangle N_G(t)\subseteq R 
$$
since $1\in N_G(r)\cap N_G(t)$.
It follows that the induced subgraphs $G_{(t)}\left[\overline{L}\right]$ and $G_{(t)}\left[R\right]$ are contained not connected by a path in $G_{(t)}$.  By applying Corollary \ref{superdiagonals are ones}, observe that $2\notin N_G(t)$ and $2\in N_G(1)$, so 
$$
N_{G^{(t)}}(1)=N_G(1)\triangle N_G(t)\ni 2,
$$
and therefore $G_{(t)}\left[\overline{L}\right]$ contains an edge but no loops.  It follows that $G^{(t)}$ does not admit a successful pressing sequence, a contraction.  Therefore $\sigma_1\not>1$.

We now proceed to show necessity: if $G\in \mathbf{CUP}_n$, then $U\in \mathcal{M}_n$.
Let $G\in \mathbf{CUP}_n$ with instructional Cholesky root $U=(u_{i,j})$.  By Lemmas \ref{pressed graph has nice cholesky} and \ref{principal submatrix lemma} we have that $G_{\boldsymbol{1}}, G-n \in \mathbf{CUP}_{n-1}$ and therefore by the inductive hypothesis $U_{\hat{1}}, U_{\hat{n}} \in \mathcal{M}_{n-1}$.  For simplicity we let $H=G-n$ and $w_i = \weight_G(i)$ for $i \in V(G)$ throughout the rest of this proof.  
\noindent It suffices to show the following four conditions hold for $U$:
\begin{itemize}
\item[(I)] Property 1 holds for the $n^{\textrm{th}}$ column,
\item[(II)] $w_n\geq w_{n-1}$,
\item[(III)] If $w_{n-2}>2$ then $w_n>w_{n-2}$,
\item[(IV)] If $w_n\neq n$ then the first column of $M$ is the only one with odd weight.
\end{itemize}
\noindent We have four cases to consider.\\

\noindent \textbf{First Case: $u_{1,n}=u_{2,n}=1$.}
Since $u_{2,n}=1$ then Property 4 of $U_{\hat{1}}$ gives us $u_{i,n}=1$ for all $2\leq i\leq n$.  It follows that $w_n=n$ and therefore (I), (II), (III), and (IV) hold.\\

\noindent \textbf{Second Case:} $u_{1,n}=u_{2,n}=0$. 
(I)  holds by Property 1 of $U_{\hat{1}}$.  Recall that the diagonal and super-diagonal entries of $U$ must be $1$ by Corollary \ref{superdiagonals are ones} so we need not consider the case where $n\leq 3$.  For $n=4$ we have two matrices to consider, 
$$
V_1=\begin{bmatrix}
1 & 1 & 0 & 0\\
0 & 1 & 1 & 0\\
0 & 0 & 1 & 1\\
0 & 0 & 0 & 1\\
\end{bmatrix}\quad \textrm{ and }\quad V_2=\begin{bmatrix}
1 & 1 & 1 & 0\\
0 & 1 & 1 & 0\\
0 & 0 & 1 & 1\\
0 & 0 & 0 & 1\\
\end{bmatrix}.
$$
$V_1$ satisfies (I) - (IV).  $V_2 \notin \mathbf{CUP}_4$ since $(3,4,1,2)$ is also a successful pressing sequence.  Thus we may assume $n\geq 5$ to show (II), (III) and (IV) hold.
\begin{claim}\label{Claim}
$u_{1,n-1}=0$
\end{claim}
\begin{myproof}{Claim}{\ref{Claim}}
Assume towards a contradiction that $u_{1,n-1}=1$. Property 1 of  $U_{\hat{n}}$ tells us that 
$$
w_{n-1}=\weight_{U_{\hat{n}}}(n-1)=n-1
$$
and so 
$$
\weight_{U_{\hat{1}}}(n-1)=w_{n-1}-u_{1,n-1}=n-2 .
$$
By Property 4 of $U_{\hat{1}}$, since $u_{2,n}=0$, then 
$$
\weight_{U_{\hat{1}}}(n-1)\equiv 0 .
$$
  It follows that 
$$
w_{n-1}=n-1\equiv 1.
$$
Recalling that $u_{1,n}=u_{2,n}=0$, by Property 2 of $U_{\hat{1}}$ we have 
$$
n-2\geq \weight_{U_{\hat{1}}}(n)\geq \weight_{U_{\hat{1}}}(n-1)=n-2
$$
and so 
$$
w_n=\weight_{U_{\hat{1}}}(n)=n-2\equiv 0. 
$$
\noindent Let 
$$
k=\min\limits_{1<i\leq n}\{i\mid w_i\equiv 1\}.
$$
Since $G\in \mathbf{CUP}_n$ and $k\neq 1$ then $G^{(k)}$ is not a pressable graph.  We will use this to arrive at a contradiction. Since $w_n\equiv 0$, 
$$
\mathcal{L}(G)=\mathcal{L}(H).
$$
 By the minimality of $k$,  by Property 4 of $U_{\hat{n}}$, and since $n-1\equiv 1$: 
$$
\mathcal{L}(H)=\{1\}\cup  \{k, k+2, \ldots, n-1\}.
$$
Observe that 
$$
\langle k,n \rangle_G=\osum_{i=1}^{n}u_{i,k}u_{i,n}= \osum_{i=1}^{2}(u_{i,k}\cdot 0)\oplus \osum_{i=3}^{k}(1\cdot 1)\oplus \osum_{i=k+1}^{n}(0\cdot u_{i,n})\equiv k-2\equiv 1 
$$
and for $i\in [n-1]$
$$
\langle k,i \rangle_G=\langle k,i \rangle_H\equiv \min(\weight_H(k),\weight_H(i))
$$
and so
$$
N_G(k)=\{1\}\cup [k, n].
$$
It follows that 
$$
\mathcal{L}(G^{(k)})=\mathcal{L}(G)\triangle N_G(k)=\{k+1, k+3, \ldots, n-2,n\}.
$$
Since $N_G(k)\cap [k-1]=\{1\}$, the only potential edge in $G\left[[k-1]\right]$ affected by pressing $k$ is the loop on $1$.  Furthermore $G\left[[k-1]\right] \in \mathbf{CUP}_{k-1}$ by Corollary \ref{principal submatrices} so we may conclude that $G^{(k)}\left[[k-1]\right]$ is connected.
\noindent If $k\neq n-1$ then $\langle k+1, n\rangle_G\equiv k-1\equiv 0$ so 
$$
N_G(k+1)=\{1,k\}
$$
 and so 
$$
N_{G^{(k)}}(k+1)=N_G(k)\triangle N_G(k+1) = [k+1, n].
$$
Then it follows that $G^{(k)}\left[[k+1, n]\right]$ is a connected graph with looped vertex $(k+1)$.  Observe that 
$$
\langle 1,n \rangle_G=0, \quad \langle 1,k \rangle_G=\langle k,n \rangle_G=1
$$
  and so 
$$
\langle 1,n \rangle_{G^{(k)}}=1
$$
Therefore, if $k\neq n-1$, $G^{(k)}$ is a connected graph with at least one looped vertex, contradicting the fact that $G\in \mathbf{CUP}_n$.  We must conclude that $k=n-1$.  Then $N_G(k)=\{1,n-1, n\}$ and so pressing $k$ only affects four pairs of vertices:
$$
\{(1,1),(1,n),(n-1,n-1), (n,n)\}.
$$
Once again, since $G\left[[k-1]\right]$ is connected, $G^{(k)}\left[[k-1]\right]$ must be connected as well (although possibly without loops).  However
$$
\langle n,n \rangle_G = 0, \quad \langle 1,n \rangle_G = 0, \textrm{ and }\quad \langle 1,k \rangle_G=\langle k,n\rangle_G=1
$$
 so 
$$
\langle 1,n \rangle_{G^{(k)}}=1 \quad \textrm{ and }\quad \langle n,n \rangle_{G^{(k)}}=1.
$$
  It follows that $G^{(k)}$ is a connected graph with at least one looped vertex, namely $n$.  This implies $G^{(k)}$ is a pressable graph, contradicting that $G\in \mathbf{CUP}_n$. Claim \ref{Claim} is established. 
\end{myproof}\newline

\begin{claim}\label{Claimm}
$u_{1,n-2}=0$
\end{claim}
\begin{myproof}{Claim}{\ref{Claimm}}
Assume towards a contradiction that $u_{1,n-2}=1$.  By Property 1 of $U_{\hat{n}}$ 
$$
w_{n-2}=\weight_{U_{\hat{n}}}(n-2)=n-2.
$$
  By Claim 1 
$$
\weight_{U_{\hat{n}}}(n-1)<n-1
$$
 and so by Property 4  
$$
\weight_{U_{\hat{n}}}(n-2)\equiv 0
$$
 and therefore $n\equiv 0$.
\noindent We then have 
$$
\weight_{U_{\hat{1}}}(n-2)=w_{n-2}-u_{1,n-2}=n-3\equiv 1 
$$
which contradicts Property 4 of $U_{\hat{1}}$ since $u_{2,n}=0$.  This establishes Claim \ref{Claimm}.
\end{myproof}\newline

We may now assume $u_{1,n-2}=u_{1,n-1}=0$ and proceed to show (II)-(IV).  (II) follows from Property 2 of $U_{\hat{1}}$ since 
$$
w_n=\weight_{U_{\hat{1}}}(n)\geq \weight_{U_{\hat{1}}}(n-1)=w_{n-1}. 
$$
To verify (III), observe that if $w_{n-2}>2$ then 
$$
\weight_{U_{\hat{1}}}(n-2)>2 \quad \Rightarrow \quad  w_n=\weight_{U_{\hat{1}}}(n)>\weight_{U_{\hat{1}}}(n-2)=w_{n-2}
$$
 by Property 3 on $U_{\hat{1}}$.
(IV) is established by observing that 
$$
u_{1,n-1}=0 \quad \Rightarrow \quad \weight_{U_{\hat{n}}}(n-1)<n-1\quad \Rightarrow \quad  \weight_{U_{\hat{n}}}(i)\equiv 0 \textrm{ for all $2<i< n$}
$$
 by Property 4 of $U_{\hat{n}}$ and so 
$$
w_i=\weight_{U_{\hat{n}}}(i)\equiv 0 \textrm{ for all $2<i< n$}.
$$
\noindent \textbf{Third Case:} $u_{1,n}=0$,  $u_{2,n}=1$.\newline 
By Property 1 of $U_{\hat{1}}$ we have that $u_{i,n}=1$ for all $2\leq i \leq n$.  It follows that $w_n=n-1$ and so (I) and (II) hold.  Furthermore, since $w_{n-2}\leq n-2$ then (III) holds.  To verify (IV), we need to show that $\mathcal{L}(G)=\{1\}$.
\begin{claim}\label{cclaim}
$\mathcal{L}(G)\neq\{1,n\}$.
\end{claim}
\begin{myproof}{Claim}{\ref{cclaim}}
Assume, by way of contradiction, that $\mathcal{L}(G)=\{1,n\}$.  Since $G\in\mathbf{CUP}_n$ and $n\neq 1$ then $G^{(n)}$ must contain a non-trivial loopless component $C$.  Choose and fix $p\in V(C)$.  Since $C$ is non-trivial we may assume $p\neq 1$.  Since $H\in \mathbf{CUP}_{n-1}$ there must exist a path from $p$ to a looped vertex in $H$.  Since $\mathcal{L}(H)=\{1\}$ this looped vertex must be $1$.  Choose such a path $P$, 
$$
P=v_0\ldots v_\ell
$$
 where $p=v_0$ and $v_\ell=1$.  Observe that 
$$
\langle 1,n\rangle_G=0
$$
 so $1$ is a looped vertex in $G^{(n)}$ as well.  Then some interior edge of $P$ must be removed upon pressing $n$ as otherwise $p$ would have a path to a looped vertex in $G^{(n)}$.  Let 
$$
j=\min\limits_{0\leq i<\ell}\{i\mid \langle v_i,n\rangle_G=1 \}
$$
then 
$$
P'=v_0P v_j
$$
 is a path from $p$ to a looped vertex in $G^{(k)}$, a contradiction.  This establishes Claim \ref{cclaim}.
\end{myproof}\newline
\begin{claim}\label{ccclaim}
If $n\equiv 0$ then $\mathcal{L}(G)=\{1\}$.
\end{claim}
\begin{myproof}{Claim}{\ref{ccclaim}}
Let $n\equiv 0$.  Assume, by way of contradiction, that $\mathcal{L}(G)\neq \{1\}$ and let
$$
k=\min\limits_{3\leq i\leq n-1}\{i\mid w_i\equiv 1\}.
$$
Choose a non-trivial loopless component $C$ of $G^{(k)}$ and a vertex $p \in V(C)\setminus \{1\}$.
Recall that $u_{1,n}=0$ and $u_{2,n}=1$.  Then 
$$
\langle n,n \rangle_G=\osum\limits_{i=1}^n u_{i,n}\equiv n-1\equiv 1 \quad \textrm{and}\quad \langle k, n \rangle_G=\osum\limits_{i=1}^n u_{i,k}\equiv k-1\equiv 0$$ implies $$\langle n,n \rangle_{G^{(k)}}=1
$$
 so $n\neq p$. For any $j\in [k+1, n-1]\setminus \mathcal{L}(G)$, since $H \in \mathbf{CUP}_{n-1}$ and $w_k \equiv 1$, Property 3 implies that $\langle k, j \rangle_H= 1$, whence 
$$
\langle j,j \rangle_G=0 \textrm{ and }\langle k, j \rangle_G= 1\quad \textrm{ implies }\quad \langle j,j \rangle_{G^{(k)}}=1
$$
 so $j$ is looped in $G^{(k)}$ and therefore $p\notin [k+1, n-1]\setminus \mathcal{L}(G)$. If $j\in [k+1, n-1]\cap \mathcal{L}(G)$ then $j\in [k+2, n-1]$ and
$$
\langle j,j-1 \rangle_G=0 \textrm{ and }\langle k, j-1 \rangle_G=\langle k, j \rangle_G= 1\quad \textrm{ implies }\quad \langle j,j-1 \rangle_{G^{(k)}}=1
$$
 so $p\notin [k+1, n-1]\cap \mathcal{L}(G)$.  Therefore 
$$
p\in [2,k-1].
$$
Since $G\left[[k-1]\right]\in \mathbf{CUP}_{k-1}$ then there exists a path $P$ in $G\left[[k-1]\right]$ connecting $p$ to a looped vertex; since $\mathcal{L}\left(G\left[[k-1]\right]\right)=\mathcal{L}(G)\cap [k-1]=\{1\}$ then the looped vertex must $1$.  
$$
P=v_0\ldots v_{\ell}
$$
 where $v_0=p$ and $v_{\ell}=1$. Note that 
$$
N_G(k)\cap V(P)=\{1\}
$$
because, if $i = \min \{j : v_j \in N_G(k)\}$ is not $1$, then $v_i$ is looped in $G^{(k)}$ and in the same component (namely, $C$) with $p$, a contradiction. Then all the interior edges of $P$ are unaffected by pressing $k$ (although the loop on $1$ is removed).  If $2\in V(P)$ then 
$$
P'=v_0P2n
$$
 is a path from $p$ to $n$ in $G^{(k)}$.  If $2 \notin V(P)$ then 
$$
P'=v_0Pv_\ell 2n
$$
 is a path from $p$ to $n$ in $G^{(k)}$.  Since $n$ is looped in $G^{(k)}$ this contradicts that $C$ is a non-trivial loopless component.  Claim \ref{ccclaim} is established.
\end{myproof}\newline
By Claims \ref{cclaim} and \ref{ccclaim} we have only one case left to consider.  Let $n\equiv 1$ and assume, by way of contradiction, that $\mathcal{L}(G)\neq \{1\}$.  Let
$$
k=\min\limits_{3\leq i\leq n-1}\{i\mid w_i\equiv 1\}.
$$
Choose a non-trivial component $C$ of $G^{(k)}$ and a vertex $p\in V(C)\setminus\{1\}$.  Observe that $n-1\equiv 0$ so $\weight_G(n-1)\equiv 0$ by Property 4 of $H$.  Then, by Property 4 of $H$
$$
\langle n-1,n-1 \rangle_G\equiv n-1\equiv 0 \textrm{ and }\langle k, n-1 \rangle_G\equiv k\equiv 1$$ which implies $$\langle n-1,n-1 \rangle_{G^{(k)}}=1
$$
so $n-1\neq p$. 
Similarly,
$$
\langle n-1,n \rangle_G=1\cdot 0\oplus \osum_{i=2}^{n-1}1\cdot 1 \equiv n-2\equiv 1 \textrm{ and }\langle k, n \rangle_G=  1\cdot 0\oplus \osum_{i=2}^{k}1\cdot 1 \equiv k-1\equiv 0 $$ implies $$\langle n-1,n \rangle_{G^{(k)}}=1
$$
so $n\neq p$ as it is connected to a looped vertex. If $j\in [k+1, n-2]\setminus \mathcal{L}(G)$ then
$$
\langle j,j \rangle_G=0 \textrm{ and }\langle k, j \rangle_G=1 \quad \textrm{ implies }\quad \langle j,j \rangle_{G^{(k)}}=1
$$
 so $p\notin [k+1, n-2]\setminus \mathcal{L}(G)$.  If $j\in [k+1, n-2]\cap \mathcal{L}(G)$ then $j\in [k+2, n-2]$ and
$$
\langle j,j-1 \rangle_G=0 \textrm{ and }\langle k, j-1 \rangle_G=\langle k, j \rangle_G=1 \quad \textrm{ implies }\quad \langle j,j-1 \rangle_{G^{(k)}}=1
$$
 so $p\notin [k+1, n-2]\cap \mathcal{L}(G)$.  It follows that 
$$
p\in [2,k-1].
$$
Since $G\left[[k-1]\right]\in \mathbf{CUP}_{k-1}$ has only $1$ as a looped vertex then $G\left[[k-1]\right]$ contains a path 
$$
P=v_0\ldots v_{\ell}
$$
 where $v_0=p$ and $v_{\ell}=1$.  Since 
$$
N_G(k)\cap [k-1]=\{1\}
$$
 then none of the interior edges are removed upon pressing $k$.  If $2\in V(P)$ we have 
$$
P'=v_0P2n(n-1)
$$
 is a path to a looped vertex in $G^{(k)}$ and if $2\notin V(P)$ then 
$$
P'=v_0Pv_\ell 2n(n-1)
$$
 is a path to a looped vertex in $G^{(k)}$.  This contradicts that $C$ is a loopless component in $G^{(k)}$. We conclude that $\mathcal{L}(G)=\{1\}$ and therefore establish (IV).\\
 
\noindent \textbf{Fourth Case:} $u_{1,n}=1$, $u_{2,n}=0$.\\
We show that this case cannot occur.   Assume it does to reach a contradiction.
Since $u_{2,n}=0$ then by Property 4  only the initial column  of $U_{\hat{1}}$ has odd weight.  It follows that 
$$
w_n=\weight_{U_{\hat{1}}}(n)+1\equiv 1.
$$
  Let 
$$
k=\min\limits_{2\leq i\leq n}\{i\mid w_i\equiv 1\}.
$$
\begin{claim}\label{cccclaim}
$k\neq n$
\end{claim}
\begin{myproof}{Claim}{\ref{cccclaim}}
Assume, by way of contradiction, that $k=n$. Then $\mathcal{L}(G)=\{1,n\}$.  Since $G\in \mathbf{CUP}_n$ we may conclude that $G^{(n)}$ contains a non-trivial loopless component $C$.  Choose and fix $p\in V(C)\setminus \{1\}$.   Since $H=G\left[[n-1]\right]\in \mathbf{CUP}_{n-1}$ there exists a path $P$ in $H$ that connects $p$ to a looped vertex, namely $1$. Let 
$$
P=v_0\ldots v_\ell
$$
 where $v_0=p$ and $v_\ell=1$.
Observe that 
$$
\langle v_\ell, n\rangle_G=\langle 1,n\rangle_G=1
$$
 and let 
$$
m=\min\limits_{0\leq i\leq \ell}\{i\mid v_i\in N_G(n)\}.
$$
If $m<\ell$ then 
$$
P'=v_0Pv_m
$$
 is a path to a looped vertex in $G^{(n)}$, contrary to assumption.  Assume $m=\ell$ and let $q=v_{\ell -1}$. 
$$
1=\langle v_\ell, v_{\ell-1}\rangle_G=\langle 1,q\rangle_G
$$
 so $u_{1,q}=1$.  By Property 1 of $U_{\hat{n}}$, $u_{i,q}=1$ for all $i\in [q]$ and so $w_q=q$.
\noindent Let 
$$
r=\min\limits_{2\leq i\leq n}\{i\mid u_{i,n}=1\}.
$$
  Since $\langle q, n\rangle_G=\langle v_{\ell-1}, n\rangle_G=0$ then 
$$
0=\osum\limits_{i=1}^{n}(u_{i,q}u_{i,n})= (1\cdot 1) \oplus \osum\limits_{i=2}^{r-1}(u_{i,q}\cdot 0) \oplus \osum\limits_{i=r}^{q}(1\cdot 1)\oplus \osum\limits_{i=q+1}^{n}(0\cdot 1)
$$
  and therefore 
$$
\langle q, n\rangle_G = 0=1\oplus \osum\limits_{i=r}^{q}1
$$
 which implies $r\leq q$. Furthermore $q\neq n-1$ since otherwise we would have 
$$
w_{n-1}=w_q=q=n-1
$$
 and  
$$
n-1\geq w_n=u_{1,n}+\weight_{U_{\hat{1}}}(n)=1+\weight_{U_{\hat{1}}}(n)\geq 1+\weight_{U_{\hat{1}}}(n-1)=1+(n-2) 
$$
which would imply 
$$
w_n=n-1=w_q\equiv 0
$$
 contradicting that $n$ is a looped vertex. Thus $q\neq n-1$ and so $q+1<n$ which gives us that $w_{q+1}\equiv 0$ by Property 4 of $U_{\hat{1}}$ and the fact that $u_{2,n}=0$.  However,
$$
q=w_q =\weight_{U_{\hat{n}}}(q) \leq \weight_{U_{\hat{n}}}(q+1) =w_{q+1}\leq q+1
$$
 implies $w_{q+1}=q$ since $q+1\equiv 1$ and therefore $u_{1,q+1}=0$ by Property 1 of $U_{\hat{n}}$.  Then 
$$
\langle q, q+1\rangle_G= \osum\limits_{i=1}^{n} (u_{i,q}\cdot u_{i,q+1})=(1\cdot 0)\oplus \osum\limits_{i=2}^{q}(1\cdot 1)\oplus (0\cdot 1)\oplus \osum\limits_{i=q+2}^{n}(0\cdot 0)=1
$$
 and  
$$
\langle q+1, n\rangle_G= (0\cdot u_{1,n} )\oplus \osum\limits_{i=2}^{r-1}(u_{i,q+1}\cdot 0) \oplus \osum\limits_{i=r}^{q+1}(1\cdot 1)\oplus \osum\limits_{i=q+2}^{n}(0\cdot u_{i,n}) 
$$
 and so 
$$ 
\langle q+1, n\rangle_G=\osum\limits_{i=r}^{q+1}1=1\oplus \osum\limits_{i=r}^{q}1=\langle q, n\rangle_G =0 
$$
\noindent Observe that 
$$
\langle 2,q+1\rangle_G=\osum_{i=1}^{2}u_{i,q+1}=1 $$
 so 
$$
P'=q(q+1)2
$$
 is a path in $G$ whose $q(q+1)$ and $(q+1)2$ edges are unaffected by pressing $n$ (since $(q+1)\notin N_{G}(n)$) and since $$\langle 2,2\rangle_G=\osum_{i=1}^{2}u_{i,2}=0\quad \textrm{and}\quad  \langle 2,n\rangle_G=\osum_{i=1}^{2}u_{i,n}=1$$ 
$$
P^*=pPqP'2
$$
 is a path from $p$ to a looped vertex in $G^{(n)}$, contrary to assumption that $p$ is contained in a loopless component of $G^{(n)}$.  This contradicts that $G\in \mathbf{CUP}_n$ and establishes Claim \ref{cccclaim}.
\end{myproof}

\noindent We proceed under the assumption that 
$$
k=\min\limits_{2\leq i\leq n}\{i\mid w_i\equiv 1\}<n,
$$
 once again in search of a contradiction.   Observe that we need not consider the case where $n\leq 3$ since the super-diagonal entries must be all $1$.  Furthermore, if $n=4$ we have only two matrices to consider, both of which have an additional successful pressing sequence given by $(4,3,2,1)$:
$$
\begin{bmatrix}
1&1&0&1\\ 0&1&1&0\\ 0&0&1&1\\ 0&0&0&1\\
\end{bmatrix} \quad \textrm{ and }\quad \begin{bmatrix}
1&1&1&1\\ 0&1&1&0\\ 0&0&1&1\\ 0&0&0&1\\
\end{bmatrix}.
$$
Assume $n\geq 5$.  By Property 4 of $U_{\hat{n}}$ we have $u_{1,n-1}=1$ and so by Property 1 and 2 of $U_{\hat{n}}$ 
$$
n-1\geq w_n= \weight_{U_{\hat{1}}}(n)+1 \geq  \weight_{U_{\hat{1}}}(n-1)+1 =(n-2)+1=n-1
$$
 which gives us $w_n=n-1$.  Observing that $\weight_{U_{\hat{1}}}(n)=n-2$ and $u_{2,n}=0$ we conclude from Property 4 of $U_{\hat{1}}$ that $n-2 \equiv 0$, so $n$ is looped in $G$.  Since $G\in \mathbf{CUP}_n$ then $G_{(n)}$ must contain a non-trivial loopless component, say $C$.  Choose and fix a vertex $p \in V(C)\setminus \{1\}$.  Observe that $G\left[[p]\right] \in \mathbf{CUP}_p$ by Corollary \ref{principal submatrices}. 
If $p\in \mathcal{L}(G)$ then pressing $n$ must remove its loop so  $$\langle 2,p\rangle_G =0  \quad \textrm{and }\langle 2,n\rangle_G =\langle p,n\rangle_G = 1$$ which implies that $\langle 2,p\rangle_{G_{(n)}} = 1$ and $2 \in \mathcal{L}(G_{(n)})$, contradicting that $C$ is loopless.  It follows that $p\notin \mathcal{L}(G)$ and therefore, in $G\left[[p]\right]$, we have a path $P$  from $p$ to a looped vertex $b$: 
$$
P=v_0\ldots v_\ell
$$
 where $v_0=p$, $v_\ell=b$, and $v_i$ is loopless in $G\left[[p]\right]$ for $0< i\leq \ell$.  Observe that $N_G(n)\cap \{v_1, v_2,\ldots,v_{\ell} \}=\emptyset$ since otherwise we would have a looped vertex in $C$.  It follows that the interior edges of $P$ are unaffected by pressing $n$ in $G$ and therefore it must be the case that pressing $n$ in $G$ removes the loop from $b=v_\ell$.  By Property 4 on $U_{\hat{n}}$ looped vertices in $H$ have full weight.  Observe that
$$
\osum\limits_{i=1}^{n}(u_{i,b}u_{i,n})= (1\cdot 1) \oplus (u_{2,b}\cdot 0) \oplus \osum\limits_{i=3}^{n}(u_{i,b}\cdot 1)\equiv\begin{cases}
1, & \textrm{ if $b= 1$}\\
1+ (b-2), & \textrm{ if $b\neq 1$}\\
\end{cases} 
$$
Since $1=\langle b,n \rangle_G=\osum\limits_{i=1}^{n}(u_{i,b}u_{i,n})$ this implies that $b=1$.  (Otherwise, $b$ is even but looped in $G$, contradicting Property 4.)
Observe that $w_3 = 3$ would imply $w_4=\weight_{U_{\hat{n}}}(4) = 4$ by Property 4 of $U_{\hat{n}}$, and then $\weight_{U_{\hat{1}}}(4)\equiv 1$ but $u_{2,n}=0$, contradicting Property 4 of $U_{\hat{1}}$.  Then 
$$
w_3=\weight_{U_{\hat{n}}}(3)=2
$$
by Property 2 of $U_{\hat{n}}$ and so 
$$
\langle 3,3\rangle_G=0 \quad\textrm{ and }\quad  \langle 3,n\rangle_G=1 \quad \textrm{ implies }\quad \langle 3,3\rangle_{G^{(n)}}=1
$$
 and so  $3\notin V(P)$.  Furthermore 
$$
\langle 1,3\rangle_G=0 \quad\textrm{ and }\quad  \langle 1,n\rangle_G=\langle 3,n\rangle_G=1 \quad \textrm{ implies }\quad \langle 1,3\rangle_{G^{(n)}}=1.
$$
Therefore 
$$
P'=v_0Pv_\ell 3
$$
 is a path to a looped vertex in $G^{(n)}$.  This implies $G\notin \mathbf{CUP}_n$ contrary to assumption.  Therefore Case 4 cannot occur. 
\end{proof}

\section{Recognition and Enumeration}
A straightforward and very slow way to check if a simple pseudo-graph on $n$ vertices is uniquely pressable is to check the pressability of each one of its $n!$ orderings. Here we offer a substantially faster algorithm.
\begin{corollary}
The unique pressability of $G$ can be decided in time $\mathcal{O}(n^3)$.
\end{corollary}
\begin{proof}
Let $G=(V,E)$ be a simple pseudo-graph on $n$ vertices.    Let $G'=([n], E')$ be the result of relabeling of $V(G)$ so that $G'$ is order-pressable. If $G$ is uniquely pressable then the instructional Cholesky root of $G'$ is in $\mathcal{M}_n$.  Observe that for each $k\in \mathcal{L}(G')\setminus \{1\}$, 
$$
N_{G'}(k)=\{1\}\cup [k,n]
$$
since by Property 4 each $\ell \in [k,n]$ has full weight and therefore $\langle k,\ell \rangle_{G'}\equiv k\equiv 1$.  However, for each  $k\in \mathcal{L}(G')\setminus \{1\}$ and $\ell \in [k,n]$, $\langle 1,\ell \rangle_{G'}=1$ by Property 4, and 
$$\langle 1,2 \rangle_{G'}=1 \quad \textrm{ and }\quad  \langle 2,k\rangle_{G'}\equiv 2\equiv 0.$$
Thus,
$$
N_{G'}(1)\supseteq \{2\}\cup N_{G'}(k)\supsetneq N_{G'}(k).
$$
Therefore $1$ is the unique looped vertex of largest degree.  It follows that to find a (potentially unique) pressing order it suffices to iterate the process of finding the looped vertex of largest degree and pressing it.
Index the vertices of graph $G$ arbitrarily and define $A=(a_{i,j})\in \mathbb{F}_2^{n\times n}$, the adjacency matrix of the graph.  Algorithm 1 finds a successful pressing sequence for $G$ (given that one exists) by finding the looped vertex of largest degree and pressing it; this has running time $\mathcal{O}(n^3)$, as it amounts to performing in-place Gaussian elimination on an $n\times n$ matrix.  Algorithm 2 computes the instructional Cholesky root of an ordered adjacency matrix and once again is done by performing Gaussian elimination. Finally, Algorithm 3 checks if an upper-triangular matrix has the properties of $\mathcal{M}_n$ which is done by computing no more than $n$ partial sums for each of $n$ columns and comparing them sequentially.  Algorithm 3 also has running time $\mathcal{O}(n^3)$.\\

\noindent \textbf{Algorithm 1:} Find a Pressing Order
\begin{algorithmic}[1]
\STATE \textbf{input:} Adjacency matrix $A$  with entries $a_{i,j}$ for $i,j\in [n]$
\STATE \textbf{output:}  Re-indexed matrix $P^TAP$
\STATE $M\leftarrow A$ 
\STATE $P\leftarrow 0_{n \times n}$
\STATE $t\leftarrow 1$ 
\WHILE{$t \leq n$}
\STATE $maxDegree \leftarrow 0$  
\STATE $indexMaxDegree \leftarrow 0$ 
\STATE $i\leftarrow 1$
\WHILE{$i\leq n$}
\STATE $deg_i\leftarrow 0$
\IF{$m_{i,i}=1$}
\STATE $deg_i\leftarrow \sum_{j=1}^{n}m_{i,j}$
\IF{$deg_i> maxDegree$}
\STATE $maxDegree \leftarrow d_i$
\STATE $indexMaxDegree \leftarrow i$
\ENDIF
\ENDIF
\STATE $i\leftarrow i+1$
\ENDWHILE
\IF{$indexMaxDegree =0$}
\IF{$\sum_{\ell=1}^{n}\sum_{j=1}^{n}m_{\ell,j}>0$}
\RETURN{False} \COMMENT{Not a Pressable Graph}
\ENDIF
\ELSE
\STATE $k\leftarrow indexMaxDegree$
\STATE $p_{t,k}\leftarrow 1$
\STATE $t \leftarrow t+1$
\FOR{$\ell\in [n]$}
\FOR{$j \in [n]\setminus \{k\}$}
\STATE $m_{\ell,j}\leftarrow m_{\ell,j}\oplus (m_{k,j} \cdot m_{\ell,k} )$ 
\ENDFOR
\ENDFOR
\FOR{$j\in [n]$}
\STATE $m_{k,j}\leftarrow 0$
\ENDFOR
\ENDIF
\ENDWHILE
\RETURN $P^TAP$
\end{algorithmic}\textrm{ }\\

\noindent \textbf{Algorithm 2:} Construct instructional Cholesky root
\begin{algorithmic}[1]
\STATE \textbf{input:} Adjacency matrix $A$  with entries $a_{i,j}$ for $i,j\in [n]$
\STATE \textbf{output:} Instructional Cholesky matrix $U$
\STATE $U\leftarrow 0_{n \times n}$
\STATE $k\leftarrow 1$
\WHILE{$k\leq n$}
\IF{$a_{k,k}=1$}
\FOR{$j\in [n]$}
\STATE $u_{k,j}\leftarrow a_{k,j}$
\ENDFOR
\FOR{$i\in [k+1,n]$}
\FOR{$j\in [n]$}
\STATE $a_{i,j}\leftarrow a_{i,j}\oplus (a_{k,j} \cdot a_{i,k})$
\ENDFOR
\ENDFOR
\STATE $k\leftarrow k+1$
\ELSE 
\STATE $k \leftarrow n+1$
\ENDIF
\ENDWHILE
\RETURN $U$
\end{algorithmic} \textrm{ }\\
\noindent \textbf{Algorithm 3:} Does this instructional Cholesky correspond to a uniquely pressable OSP?
\begin{algorithmic}[1]
\STATE \textbf{input:} Instructional Cholesky matrix $U$  with entries $u_{i,j}$ for $i,j\in [n]$
\STATE \textbf{output:}  True or False
\STATE $j\leftarrow 1$
\WHILE{$j\leq n$}
\STATE $i\leftarrow 1$
\WHILE{$i\leq j$}
\IF{$u_{i,j}=0$}
\STATE $i\leftarrow i+1$
\ELSIF{$\sum_{\ell=i}^{j}u_{\ell,j}< j-i+1$}
\RETURN False
\ENDIF
\ENDWHILE
\STATE $j\leftarrow j+1$
\ENDWHILE
\STATE $j \leftarrow 1$
\WHILE{$j<n$}
\IF{$\sum_{\ell=1}^{j}u_{\ell,j}>\sum_{\ell=1}^{j+1}u_{\ell,j+1}$}
\RETURN False
\ELSIF{$j\leq n-2$ and $\sum_{\ell=1}^{j}u_{\ell,j}>2$ and $\sum_{\ell=1}^{j+2}u_{\ell,j+2}=\sum_{\ell=1}^{j}u_{\ell,j}$}
\RETURN False
\ELSIF{$\sum_{\ell=1}^{j+1}u_{\ell,j+1}\equiv 1$ and $\sum_{\ell=1}^{j+1}u_{\ell,j+1}\neq j+1$}
\RETURN False
\ELSE
\STATE $j\leftarrow j+1$
\ENDIF
\ENDWHILE
\RETURN True
\end{algorithmic}
\end{proof}

\begin{corollary}\label{extending right}
Let $n>0$.  Let $G=([n], E)$ and let $H=([n+1], E(H))$ be the result of adding a vertex $``n+1"$ adjacent to each looped vertex in $G$, with a loop at $n+1$ if and only if $n$ is even.  If $G\in \mathbf{CUP}_n$, then $H\in \mathbf{CUP}_{n+1}$.
\end{corollary}
\begin{proof}
Suppose $G\in \mathbf{CUP}_n$ with adjacency matrix $A$ and instructional Cholesky root $U=(u_{i,j})$.  Observe that $H=([n+1], E(H))$ where 
$$
E(H)=\begin{cases}
E(G)\cup \{(i,n+1)\mid i\in \mathcal{L}(G) \}\cup \{(n+1, n+1)\}, & \textrm{ if $n\equiv 0$ }\\
E(G)\cup \{(i,n+1)\mid i\in \mathcal{L}(G) \}, & \textrm{ if $n\equiv 1$ }\\
\end{cases}
$$
Let 
$$
V=\left[\begin{array}{c|c}
\begin{matrix} &&\\&   \mbox{{$U$}}&\\
    &&\\ \end{matrix}  & \begin{matrix} 1 \\ \vdots  \\ 1\\
\end{matrix}\\ \hline 
\begin{matrix} 0 &\cdots & 0 \\
\end{matrix}  &  1\\
\end{array}\right] 
$$
and observe that
$$
V^T\cdot V=\left[\begin{array}{c|c}
\begin{matrix} &&\\&   \mbox{{$A$}}&\\
    &&\\ \end{matrix}  & \begin{matrix} b_{1,n+1} \\ \vdots  \\ b_{n,n+1}\\
\end{matrix}\\ \hline 
\begin{matrix} b_{n+1, 1} &\cdots & b_{n+1,n} \\
\end{matrix}  &  b_{n+1, n+1}\\
\end{array}\right]
$$
  where $b_{n+1, i}=1$ if and only if $\weight_{V}(i)\equiv 1$ if and only if $  i\in \mathcal{L}(G) $ or $i=n+1 \equiv 1$.
It follows that $V^TV$ is a Cholesky factorization for the adjacency matrix of $H$.  Since $G\in \mathbf{CUP}_n$ then $U$ has $1$'s along the diagonal and so $U$ and $V$ are full rank matrices. 
Since Cholesky factorizations are unique for full-rank matrices \cite{cooper2016successful} then $V$ must be the instructional Cholesky root of $H$. 
$V$ inherits the properties of $\mathcal{M}_n$ in its first $n$ columns from $U$. Furthermore the last column of $V$ has full weight $n+1$, so $V\in \mathcal{M}_{n+1}$, and therefore $H\in \mathbf{CUP}_{n+1}$.
\end{proof}

\begin{corollary}\label{extending left}
Let $n>0$.  Let $G=([2,n+1], E)$ and $H=([n+1],E(H) )$ be the result of removing all the edges, including loops, from $G[\mathcal{L}(G)]$, adding a looped vertex `` $1$'' adjacent to all of $\mathcal{L}(G)$. If $G\in \mathbf{CUP}_n$ then $H\in \mathbf{CUP}_{n+1}$.
\end{corollary}
\begin{proof}
Let $G\in \mathbf{CUP}_n$, recall that by Property 4 (of its instructional Cholesky root) the looped vertices in $G$ form a clique.  The only looped vertex in $H$ is $1$ so if $H$ admits a successful pressing sequence it must begin with $1$. However  $H^{(1)}=G$ since pressing and deleting $1$ creates an edge between any two vertices in $\mathcal{L}(G)$.  It follows that $H$ has exactly one successful pressing sequence: $\mathbf{n+1}$.  $H\in \mathbf{CUP}_{n+1}$.
\end{proof}
\begin{notation}
$\mathbf{CUP}_{[n]}$ is the  set of connected, uniquely pressable ordered ($<_{\mathbb{N}}$) simple pseudo-graphs on vertex set $[n]$.    
\end{notation}
\begin{corollary}The number of connected, uniquely pressable simple pseudo-graphs on $n>1$ vertices up to isomorphism is
$$
\left|\mathbf{CUP}_{[n]}\right|=\begin{cases}3^{(n-2)/2}, & n \textrm{ is even}\\2\cdot 3^{(n-3)/2}, & n \textrm{ is odd}\\ \end{cases}
$$
\end{corollary}
\begin{proof}
For $n=2$, the result holds since 
$$
\mathbf{CUP}_{[2]}=\left\{\left([2], \{(1,1), (1,2)\}\right)\right\}.
$$
We proceed by induction.  Let $n\geq 2$ be even and assume $\left|\mathbf{CUP}_{[n]}\right|=3^{(n-2)/2}$.  Choose and fix $G=([n], E)\in \mathbf{CUP}_{[n]}$ with adjacency matrix $A$ and instructional Cholesky root $U=(u_{i,j})$. Let $G'=([2,n+1], E')$ be a re-indexing of $G$ given by $i\mapsto i+1$ for all $i\in [n]$.
Let $H_1=([n+1], E_1)$, $H_2=([n+1], E_2)$ where 
$$
E_1=E\cup \{(i,n+1)\mid i\in \mathcal{L}(G) \}\cup \{(n+1, n+1)\}
$$
 and 
$$
E_2=E'\triangle \left(\left( \mathcal{L}(G')\cup \{1\}\right) \times \left(\mathcal{L}(G')\cup \{1\}\right)\right).
$$
By Corollaries \ref{extending right} and \ref{extending left}, $H_1, H_2\in \mathbf{CUP}_{[n+1]}$. $H_1$ has at least two looped vertices ($1$ and $n+1$) and $H_2$ has only one looped vertex ($1$), so $H_1\neq H_2$.  Furthermore, the instructional Cholesky roots of $H_1$ and $H_2$ include $U$ as a principal submatrix on consecutive rows and columns so it is not possible that we would have gotten $H_1$ or $H_2$ by applying Corollaries \ref{extending right} or \ref{extending left} to other graphs in $\mathbf{CUP}_{[n]}$.  It follows that 
$$
\left|\mathbf{CUP}_{[n+1]}\right|\geq 2\cdot \left|\mathbf{CUP}_{[n]}\right|=2\cdot 3^{(n-3)/2}.
$$
Consider now any $H\in \mathbf{CUP}_{[n+1]}$ with instructional Cholesky root $V=(v_{i,j})$.  If $v_{1,n+1}=0$ then let $G'=H^{(1)}$ and let $G$ be the re-indexing of $V(G)$ given by $i\mapsto i-1$ for all $i\in [2,n+1]$.  By Lemma \ref{pressed graph has nice cholesky}, $G\in  \mathbf{CUP}_{[n]}$ and therefore $H$ can be constructed from $G$ using Corollary \ref{extending left}.  If $v_{1,n+1}=1$ then $v_{i, n+1}=1$ for all $i\in [n+1]$ because of Property 4 by appeal to Theorem \ref{BIG THEOREM}.  Let $G=H-\{n+1\}$. By Lemma \ref{principal submatrix lemma}, $G\in \mathbf{CUP}_{[n]}$ and hence $H$ can be obtained by applying Corollary \ref{extending right} to $G$.  It follows that 
$$
\left|\mathbf{CUP}_{[n+1]}\right|\leq  2\cdot \left|\mathbf{CUP}_{[n]}\right|=2\cdot 3^{(n-3)/2}.
$$
We count $\left|\mathbf{CUP}_{[n+2]}\right|$ in a similar, though slightly more complicated, manner.
Given $G\in \mathbf{CUP}_{[n]}$, let $H_1$ be result of two successive applications of Corollary \ref{extending right}.  Let $H_2$ be the result of applying Corollary \ref{extending right} followed by Corollary \ref{extending left}, let $H_3$ be the result of applying Corollary \ref{extending left} followed by Corollary \ref{extending right}, and let $H_4$ be the result of applying Corollary \ref{extending left} twice successively.  We first show that $H_2=H_3$ and then argue that $H_1, H_2, H_4$ are pairwise distinct.
Let $V_2$ and $V_3$ be the instructional Cholesky roots of $H_2$ and $H_3$, respectively.  Then
$$
V_2=\left[\begin{array}{c|c|c}
1 &\begin{matrix} b_{1,2} &\cdots & b_{1,n}  \\
\end{matrix}  &  1\\ \hline 
\begin{matrix}
0 \\ \vdots \\ 0 \\
\end{matrix}
&\begin{matrix} &&\\&   \mbox{{$U$}}&\\
    &&\\ \end{matrix}  & \begin{matrix} 1 \\ \vdots  \\ 1\\
\end{matrix}\\\hline  
0 &\begin{matrix} 0 &\cdots & 0\\
\end{matrix}  &  1\\
\end{array}\right]=V_3 
$$
where $b_{1,i}=1$ whenever it is positioned above a column of odd weight in $U$.  Therefore $H_2=H_3$.  Observe that $H_1$ has at least two looped vertices, $1$ and $n+1$, whereas $H_2$ and $H_4$ have only one looped vertex ``$1$''. Since $n$ is even, vertex $n$ in $G$ is loopless by Property 4.  Then the instructional Cholesky root of $H_4$ has the form 
$$
V_4=\left[\begin{array}{c|c|c}
1 &1&\begin{matrix} b_{1,3} &\cdots & b_{1,n+2}  \\
\end{matrix} \\ \hline
0 &1&\begin{matrix} b_{2,3} &\cdots & b_{2,n+2}  \\
\end{matrix} \\ \hline
\begin{matrix}
0 \\ \vdots \\ 0 \\
\end{matrix} &\begin{matrix}
0 \\ \vdots \\ 0 \\
\end{matrix}
&\begin{matrix} &&\\&   \mbox{{$U$}}&\\
    &&\\ \end{matrix} \\
\end{array}\right]
$$
where $b_{1,n+2}=b_{2,n+2}=0$ and so $H_4\neq H_2$.  It follows that 
$$
\left|\mathbf{CUP}_{[n+2]}\right|\geq 3\cdot \left|\mathbf{CUP}_{[n]}\right|=3^{((n+2)-3)/2}.
$$
Choose and fix $H\in \mathbf{CUP}_{[n+2]}$ with instructional Cholesky root $V=(v_{i,j})$.  Let  $G_1=\left(H-\{n+2\}\right)-\{n+1\}$, $G'_2=H^{(1)}-\{n+2\}$, and  $G'_3=H^{(1,2)}$. Let $G_2$ and $G_3$ be (order-preserving) re-indexings of $G'_2$ and $G'_3$ so that $V(G_2)=V(G_3)=[n]$.   By (repeated) applications of Lemmas \ref{pressed graph has nice cholesky} and  \ref{principal submatrix lemma}; $G_1, G_2, G_3 \in \mathbf{CUP}_{[n]}$.  Let $\alpha=v_{1,n+1}+v_{1,n+2}$.  If $\alpha=2$ then $v_{1,n+1}=v_{1,n+2}=1$ and by Property 1, $v_{i,n+1}=1$ for all $i\in [n+1]$ and $v_{i,n+2}=1$ for all $i\in [n+2]$.  Then $H$ can be constructed from two applications of Corollary \ref{extending right} to $G_1$. 
If $\alpha=1$ then $v_{1, n+1}=0$ and $v_{1, n+2}=1$ by Property 4.  Furthermore, Property 4 implies that $1$ is the only looped vertex, since $n+2$ is even and $v_{1, n+1}=0$.  Observing that $n+2$ must be a full weight vertex, we conclude that $H$ can be constructed from $G_2$ by application of Corollaries \ref{extending right} and \ref{extending left} (in either order).  Finally if $\alpha=0$ then $v_{1,n+2}=0$ and by Property 4, and since $n$ is even, $v_{2,n+2}=0$.  Furthermore by Property 4 it follows that $H$ and $H^{(1)}$ have each only one looped vertex.  Then $H$ can be constructed from $G_3$ by two applications of Corollary \ref{extending left}. 
\noindent It follows that 
$$
\left|\mathbf{CUP}_{[n+2]}\right|\leq 3\cdot \left|\mathbf{CUP}_{[n]}\right|=3^{((n+2)-3)/2}.
$$
Therefore 
$$
\left|\mathbf{CUP}_{[n]}\right|=\begin{cases}3^{(n-2)/2}, &\textrm{if $n$ is even}\\ 2\cdot 3^{(n-3)/2}, & \textrm{if $n$ is odd}\\ \end{cases}.
$$
\end{proof}
\begin{corollary}
The number of uniquely pressable simple pseudo-graphs on $n>1$ vertices up to isomorphism is 
$$
T_n=\begin{cases}\left(5\cdot 3^{(n-2)/2} +1\right)/2, &\textrm{if $n$ is even}\\ \left(3^{(n+1)/2} +1\right)/2, & \textrm{if $n$ is odd}\\ \end{cases}
$$
\end{corollary}
\begin{proof}
There are three non-isomorphic uniquely pressable simple pseudo-graphs on $2$ vertices: the edgeless (loopless) graph, the disconnected graph containing one looped vertex and one unlooped vertex, and the connected graph containing one looped vertex and one unlooped vertex.  We proceed by induction on $n$.  Observe that for every $k\leq n$ and for every $H\in \mathbf{CUP}_k$,  we can create a (distinct) uniquely pressable graph $G$ on $n$ vertices by adding $n-k$ isolated vertices to $H$. Similarly, if $G$ is a uniquely pressable graph, then it is either the edgeless (loopless) graph or it contains exactly one non-trivial component which must be a $\mathbf{CUP}_k$ graph for some $k\leq n$.  Hence 
$$
T_n=\sum_{k=0}^{n} \left|\mathbf{CUP}_{[k]}\right|=T_{n-1}+\left|\mathbf{CUP}_{[n]}\right|.
$$
The result follows by observing that 
$$
\frac{5\cdot 3^{((n-1)-2)/2} +1}{2}+2\cdot 3^{(n-3)/2}=\frac{5\cdot 3^{(n-3)/2} +1+4\cdot 3^{(n-3)/2}}{2}=\frac{3^{(n+1)/2} +1}{2} 
$$
and 
$$
\frac{3^{((n-1)+1)/2} +1}{2}+3^{(n-2)/2}=\frac{3\cdot 3^{(n-2)/2} +1+2\cdot 3^{(n-2)/2}}{2}=\frac{5\cdot 3^{(n-2)/2} +1}{2}
$$
\end{proof} 
\section{Conclusion}

We end with a few open problems raised by the above analysis.\\

Our recognition algorithm has $\mathcal{O}(n^3)$ running time as it relies on explicit Gaussian elimination.  In \cite{andren2007complexity} and \cite{bertolazzi2014fast} the authors give improved algorithms (with sub-cubic running time) for obtaining the reduced row echelon form of a matrix with entries from a finite field.
\begin{question} Can the detection algorithm be improved to sub-cubic running time? 
\end{question}
It is natural to go beyond unique pressability for the problems of characterization, enumeration, and recognition.
\begin{question}
Which graphs have exactly $k$ pressing sequences for $k\geq 2$?
\end{question}
In \cite{cooper2016successful}, the authors consider the set of all pressing sequences of an arbitrary OSP-graph.  Some of their results suggest that problem of counting the number of pressing sequences might be easy; for example, they show that it is a kind of relaxation of graph automorphism counting, a problem which is GI-complete \cite{mathon1979note}, and therefore at worst quasipolynomial in time complexity \cite{babai2016graph}. They also connected the enumeration of pressing sequences with counting perfect matchings, a problem which is famously \#P-hard \cite{valiant1979complexity}.  Some of our results also suggest that this counting problem might be hard; for example, pressing sequences are certain linear extensions of the directed acyclic graphs whose adjacency matrices are given by instructional Cholesky roots, and enumerating linear extensions is also famously \#P-hard \cite{brightwell1991counting}.  Therefore we ask the following.
\begin{question}
What is the complexity of counting the number of pressing sequences of a graph?
\end{question}
\bibliographystyle{abbrv}
\bibliography{bibliography.bib}
\end{document}